\documentclass[12pt,reqno]{amsart}

\usepackage{booktabs}                                 
\usepackage[english]{babel}
\usepackage{latexsym}
\usepackage[utf8]{inputenc}
\usepackage{babel}
\usepackage{fancyhdr}
\usepackage{longtable}
\usepackage{mathrsfs}
\usepackage{geometry}
\usepackage{comment}
\usepackage{textcomp}
\usepackage{caption}
\usepackage{lmodern}
\usepackage{mathrsfs}
 \usepackage[T1]{fontenc}

\usepackage[all,cmtip]{xy}
\usepackage{epsfig}

\usepackage{amsmath,amsfonts,amssymb,amsthm}
\usepackage{graphics}
\usepackage{graphicx}
\usepackage{verbatim}
\usepackage{dsfont}
\usepackage{enumerate}  
\usepackage{pdflscape}
\usepackage{longtable}
\usepackage{booktabs}
\usepackage{colortbl}

\usepackage{tikz-cd}
\usepackage{ytableau}

\usepackage[shortlabels]{enumitem}
\usepackage[breaklinks]{hyperref} 
\hypersetup{
	colorlinks,
	linkcolor={red},
	citecolor={blue},
	urlcolor={red},
}
\usepackage{stmaryrd}
\usepackage{multirow}
\usepackage{longtable}
\usepackage{mathtools}

\usepackage{xcolor,colortbl}
\definecolor{lavender}{rgb}{0.9, 0.9, 0.98}

\usepackage[disable]
{todonotes}

\newcommand{\Z}{\mathbb{Z}}

\newcommand{\mQ}{\mathcal{Q}}

\newcommand{\mF}{\mathcal{F}}

\newcommand{\PP}{\mathbb{P}}

\newcommand{\KK}{\mathbb{K}}

\newcommand{\mE}{\mathcal{E}}
\newcommand{\mU}{\mathcal{U}}

\newcommand{\mZ}{\mathscr{Z}}
\newcommand{\of}{\mathcal{O}}

\newcommand{\bbGr}{\mathbb{G}\mathrm{r}}
\newcommand{\rk}{\mathrm{rk}}

\DeclareMathOperator{\D}{\textnormal{D}}

\DeclareMathOperator{\Gr}{Gr}

\DeclareMathOperator{\Bl}{Bl}

\DeclareMathOperator{\DC}{\textnormal{\textbf{D}}}

\newtheorem{thm}{Theorem}[section]
\newtheorem{corollary}[thm]{Corollary}

\newtheorem{lemma}[thm]{Lemma}
\newtheorem{proposition}[thm]{Proposition}

\newtheorem*{aim*}{Aim of this paper}

\theoremstyle{definition}
\newtheorem{definition}[thm]{Definition}
\newtheorem{rmk}[thm]{Remark}

\usepackage{thmtools}
\declaretheoremstyle[
spaceabove=1.5ex, spacebelow=1.5ex,
headfont=\bf,
notefont=\mdseries, notebraces={(}{)},
bodyfont=\normalfont,
headpunct=.,
numberwithin=,
postheadhook=\leavevmode%
qed={$\bullet$}
]{mystyle}

\usepackage[capitalise,nameinlink]{cleveref}
\crefname{mystyle}{Fano}{Fanos}
\crefname{claim}{Claim}{Claims}
\crefname{rmk}{Remark}{Remarks}
\crefname{workhyp}{WH}{WH}
\crefname{thm}{Theorem}{Theorems}
\crefname{proposition}{Proposition}{Propositions}
\crefname{app}{Appendix}{Appendices}

\crefname{lemma}{Lemma}{Lemmas}
\crefname{alg}{Algorithm}{Algorithms}
\crefname{ex}{Example}{Examples}
\makeatletter    
\def\l@subsection{\@tocline{1}{0,2pt}{2pc}{8mm}{\ \ }} 
\makeatletter    
\def\l@section{\@tocline{1}{0,2pt}{2pc}{8mm}{\ \ }} 

\newcommand{\crefpart}[2]{%
  \hyperref[#2]{\namecref{#1}~\labelcref*{#1}~(\ref*{#2})}%
}

\author{Federico Tufo}
\address{Dipartimento di Matematica \\
Universit\`a di Bologna\\
Piazza di Porta San Donato 5\\
40127 Bologna, Italy}
\email{federico.tufo96@gmail.com}

\title{On the Fano dimension of an Enriques surface}

\begin{document}
\begin{abstract}
   We construct a family of Fano fourfolds with the derived category of coherent sheaves of a general Enriques surface as semiorthogonal component. This improves a result of Kuznetsov, lowering the Fano dimension of a general Enriques surface from six to four.
\end{abstract}

\maketitle

\section{Introduction}

The Fano-visitor problem is a natural, yet elusive question in algebraic geometry and derived categories. First posed by Bondal in 2011, it amounts to asking if for any smooth projective variety $X$, there exists a smooth Fano variety $Y$ together with a full and faithful functor $\DC(X)\to\DC(Y)$. If the answer is positive $X$ is called Fano-visitor, and $Y$ is called Fano-host of $X$. This led, in \cite{KKLL17}, to the definition of the Fano dimension of a variety $X$, which is the minimal dimension of Fano-hosts of $X$.

Let us recap some key results on the topic. Bondal and Orlov in \cite{BO95} proved that the derived category of a hyperelliptic curve $X$ of genus $g$ can be embedded into the derived category of the intersection of two
quadrics in $\PP^{2g+1}$. Kuznetsov in \cite{Kuz10} proved that the derived categories of some K3 surfaces are embedded into special cubic fourfolds. Bernardara, Bolognesi, and Faenzi in \cite{BBF16} proved that every smooth plane curve is a Fano-visitor. Segal and Thomas in \cite{ST18} proved that a general quintic threefold is a Fano-visitor in an eleven-dimensional Fano-host. Finally, Kiem, Kim, H. Lee, and K. Lee, in \cite{KKLL17}, proved that all smooth projective complete intersections are Fano-visitors. 

In \cite{K18}, Kuznetsov first showed that a certain divisorial family in the moduli space of Enriques surfaces (so-called \emph{nodal Enriques}\footnote{Despite the name, these Enriques surfaces are smooth.}) can be realized as Fano-visitor for a Fano fourfold, which can be written as the blow-up of $\Gr(2,4)$ in the same Enriques surface. Then he showed that an Enriques surface general in moduli can be seen as a Fano-visitor for a six-dimensional Fano-host. In summary
\begin{itemize}
    \item if $S$ is a \emph{nodal Enriques} surface, then its Fano dimension is $\leq 4$;
    \item if $S$ is a general Enriques surface, then its Fano dimension is $\leq6$.
\end{itemize}

In our work, we improve this bound for general Enriques surfaces.
\begin{thm} [\cref{thm:main}, \cref{cor:main}]
    Let $S$ be a general Enriques surface. Then the Fano dimension of $S$ is $4$, i.e.\ there exists a Fano fourfold $Y$ with
    \[
        \DC(Y)=\langle\DC(S),E_1,\ldots,E_9\rangle
    \]
    where $E_1,\ldots,\ E_9$ are exceptional bundles. The Hodge diamond of $Y$ is diagonal, and $K_0(Y)$ contains a 2-torsion class; in particular, $\DC(Y)$ does not have a full exceptional collection.

    Such a Fano fourfold $Y$ can be realized as a complete intersection of multidegrees $(1,2,0), (1,0,2)$ in $\PP^2 \times \PP^2 \times \PP^2$. More precisely
    
    \[
        Y\cong\Bl_{S}(\PP^2\times\PP^2).
    \]
\end{thm}
Kuznetsov's construction and ours are deeply intertwined. We describe this relation in \cref{section3}.

Moreover, in \cref{prFanoHost}, we prove a general result that solves the Fano-visitor problem for certain smooth degeneracy loci. 

\subsection*{Notation.} In this paper, all varieties are defined over a number field $\KK$ of characteristic zero. We denote a $\KK$-vector space of dimension $n$ as $V_n$. The Grassmannian $\Gr(k,n)$ is the variety parametrizing subspaces $V_k\subset V_n$. On $\Gr(k,n)$, we have the Euler sequence
\[
    0\to\mU_{\Gr(k,n)}\to\of_{\Gr(k,n)}\otimes V_n\to\mQ_{\Gr(k,n)}\to 0,
\]
where $\mU_{\Gr(k,n)}$ is the tautological bundle of rank $k$ and $\mQ_{\Gr(k,n)}$ is the quotient bundle of rank $n-k$, with ample determinant. 

Let $X$ be a variety, and let $\mathcal{E}$ be a vector bundle on $X$ of rank $e$. For every integer $k\leq e$, we denote $\bbGr_X(k,\mE)$ the Grassmann bundle of $\mE$, which parametrizes $k$-dimensional vector subspaces in the fibers of $\mE$. For $\bbGr_X(k,\mE)$ one has the relative Euler sequence
\[
    0\to\mU_{\mE}\to p^*\mE\to\mQ_{\mE}\to 0
\]
where $p$ is the natural projection $p\colon\bbGr_X(k,\mE)\to X$, and $\mU_{\mE}$ the relative tautological bundle for which $p_*\mU_{\mE}=\mE$.

Let $\mE$ and $\mF$ be two vector bundles on $X$ of rank $e$ and $f$ respectively. If $\varphi:\mE\to\mF$ is a morphism of vector bundles, we denote with $\D_r(\varphi)\subset X$ the locus where rank$(\varphi_x) \leq r$, which is the $(\min(e,f)-r)$-th degeneracy locus of $\varphi$. If $\mF$ is a globally generated vector bundle on a smooth variety $X$, we denote with $\mZ(\mF)\subset X$ the zero locus of a general global section of $\mF$ in $X$, and we denote with $V(\varphi)\subset X$ the zero locus of the global section $\varphi$. 

We denote with $\DC(X)$ the bounded derived category of coherent sheaves on $X$.

\subsection*{Acknowledgements.} The author wants to thank Enrico Fatighenti, Alexander Kuznetsov, Claudio Onorati, and Andrea Petracci for the precious comments and hints on the first draft of this work.
This research has been partially funded by the European Union - NextGenerationEU under the National Recovery and Resilience Plan (PNRR) - Mission 4 Education and research - Component 2 From research to business - Investment 1.1 Notice Prin 2022 - DD N. 104 del 2/2/2022, from title ``Symplectic varieties: their interplay with Fano manifolds and derived categories", proposal code 2022PEKYBJ – CUP J53D23003840006.
The author is a member of INDAM-GNSAGA.

\section{The geometric construction} \label[section]{section2}

\subsection*{Degeneracy loci.}
In order to construct our example, we first introduce some useful tools for the study of degeneracy loci of morphisms of vector bundles. Most of the results are classical, e.g.,  \cite{fulton,WJ,DFT,BFMT,FTT}. 

\begin{thm}\label{thm:ddl}\cite[Theorem 14.4]{fulton} Let $X$ be a smooth projective variety, let $\mathcal{E}$ and $\mathcal{F}$ vector bundles on $X$ such that $\mathcal{E}^{\vee}\otimes\mathcal{F}$ is globally generated, and let $\varphi\in H^0(X,\mathcal{E}^{\vee}\otimes\mathcal{F})$ a general global section. If $\D_r(\varphi)\neq\varnothing$, then the dimension of $\D_r(\varphi)$ is $m_r=dim\ X -(\rk\ \mathcal{E}-r)(\rk\ \mathcal{F}-r)$.
\end{thm}

With the above notation, let $p:\bbGr_X(k,\mE)\to X$ be the natural projection to $X$ and assume $k\leq\rk\ \mE$. By the projection formula, we have  
\begin{equation*}
    p_*(p^*\mF\otimes\mU^{\vee}_{\mE})\cong\mF\otimes p_*(\mU^{\vee}_{\mE})\cong\mF\otimes\mE^{\vee}.
\end{equation*}
Moreover, by the Leray spectral sequence, we obtain
\begin{equation}\label{congPD1}
    H^0(\bbGr(k,\mE),p^*\mF\otimes\mU^{\vee}_{\mE})\cong H^0(X,\mE^{\vee}\otimes\mF).
\end{equation}
In particular, for a general global section $\varphi\in H^{0}(X,\mE^{\vee}\otimes\mF)$, there exists a unique global section $\varphi_{\mE}\in H^0(\bbGr(k,\mE),p^*\mF\otimes\mU_{\mE}^{\vee})$ via the isomorphism (\ref{congPD1}). Notice that $\varphi_{\mE}$ is given by the composition $\mU_{\mE}\hookrightarrow p^*\mE\xrightarrow{p^*\varphi}p^*\mF$. Analogously, let $q:\Gr_X(k,\mF^{\vee})\to X$ is the natural projection on $X$, and assume $k\leq\rk\ \mF$. Then we have 
\begin{equation}\label{eq:isovb}
    H^0(\bbGr(k,\mF^{\vee}),q^*\mE^{\vee}\otimes\mU^{\vee}_{\mF^{\vee}})\cong H^0(X,\mE^{\vee}\otimes\mF).
\end{equation}
In particular, there exists a unique $\varphi_{\mF^{\vee}}\in H^0(\bbGr(k,\mF^{\vee}),q^*\mE^{\vee}\otimes\mU^{\vee}_{\mF^{\vee}})$ associated to $\varphi$, given by the composition $q^*\mE\xrightarrow{q^*\varphi}q^*\mF\twoheadrightarrow\mU^{\vee}_{\mF^{\vee}}$.

Here, we introduce a useful notion for what follows.

\begin{definition}
    Let $X$ be a variety. A stratified Grassmann bundle on $X$ with general fiber $\Gr(k,n)$ is a pair $(Y,p)$, where $Y$ is a variety and $p$ is a surjective morphism $p:Y\to X$ such that $X$ admits a stratification $X=\bigsqcup_iX_i$, for which the morphisms $p^{-1}(X_i)\to X_i$ are locally trivial $\Gr(k,n+i)$-bundles for all $i$. If $k=1$, we call it a stratified projective bundle.
\end{definition}

The next result is a combination of \cite[Lemma 2.1]{kuznetsovKuchle} and classical facts, but we provide a full proof because we were unable to find it in this form elsewhere.

\begin{lemma}\label[lemma]{prop:degSwap}
    Let $X$ be a smooth variety. Consider
    \[
        \varphi:\mathcal{E}\to\mathcal{F}
    \]
    a general morphism of vector bundles of $\rk\ \mE=e>f=\rk\ \mF$ on $X$ such that $\mathcal{E}^{\vee}\otimes\mathcal{F}$ is globally generated. Let $p:\bbGr_X(k,\mE)\to X$ be the natural projection and consider
    \[
        \varphi_{\mE}\in H^0(\bbGr_X(k,\mE),p^*\mF\otimes\mU_{\mE}^{\vee}),  
    \]
    the global section of $p^*\mF\otimes\mU_{\mE}^{\vee}$ defined as the composition
    \[
        \mU_{\mE}\hookrightarrow p^*\mE\xrightarrow{p^*\varphi}p^*\mF.
    \]
    Assume that $\D_{f-1}(\varphi)$ is non-empty. Then the following statements hold.
    \begin{enumerate}
        \item[i.] If $k=e-f$, the variety
        \[
            Y_1=V(\varphi_{\mE})\subset\bbGr_{X}(e-f,\mE)
        \]
        is birational to $X$ via the restriction of the natural projection $p$.
        Moreover, the exceptional locus of $p_{|Y}$ on $X$ is $D_{f-1}(\varphi)$.
        \item[ii.] If $k=1$, the variety
        \[
            Y_2=V(\varphi_{\mE})\subset\PP_X(\mE)
        \]
        is a stratified projective bundle on $X$ with general fiber $\PP^{e-f-1}$. 
        Let $q:\PP_X(\mF^{\vee})\to X$ be the natural projection, and consider
        \[
             \varphi_{\mF^{\vee}}\in H^0(\PP_X(\mF^{\vee}),q^*\mE^{\vee}\otimes\of_{\mF^{\vee}}(1)),
        \]
        the global section of $q^*\mE^{\vee}\otimes\of_{\mF^{\vee}}(1)$ obtained as the composition
        \[
            q^*\mE\xrightarrow{q^*\varphi}q^*\mF \twoheadrightarrow\of_{\mF^{\vee}}(1).
        \]
        Then, $\D_{f-1}(\varphi)$ is birational to
        \[
             Z=V(\varphi_{\mF^{\vee}})\subset\PP_X(\mF^{\vee}).
        \]
        \item[iii.] With the above notation, if $k=e-f$ and $f=e-1$, then $Y_1=Y_2$ and we call it $Y$. Then, 
        \[
            Y=\Bl_{\D_{f-1}(\varphi)}X,
        \]
        and $\D_{f-1}(\varphi)$ is birational to $Z$.
        
    \end{enumerate}
\end{lemma}
\begin{proof}
    
        Let us start by proving $(i)$. By definition, $\bbGr_X(e-f,\mE)$ is the variety that parametrizes the pairs $(x,V_x)$, with $x\in X$ and $V_x$ a $(e-f)$-dimensional subspace of $\mE_x$. A point $y\in Y_1$ is a pair $(x,V_x)\in \bbGr_X(e-f,\mE)$, such that $\varphi_{\mE,x}(V_x)=0$. Notice that, by isomorphism (\ref{congPD1}), the points $x\in X$ for which $\varphi_{\mE,x}(V_x)=0$ are the points in $X$ for which $\varphi_{x}(V_x)=0$.
        Thus, the fiber of $p_{|Y_1}$ is the set of all the $V_x\subset\ker\ \varphi_x$ of dimension $e-f$ in $\ker\varphi_x$, which corresponds to the Grassmannian $\Gr(e-f,\dim(\ker\varphi_x))$. This implies that generically the fiber $p_{|Y_1}^{-1}(x)$ is a point, because generically, $\rk \ \varphi_x=f$, hence $\dim(\ker\varphi_x)=e-f$. Notice that $X$ has a stratification induced by  
        \[
            X=\D_{f}(\varphi)\supset\D_{f-1}(\varphi)\supset\D_{f-2}(\varphi)\supset\dots
        \]
        Then, over the $X_i=\D_{f-i}(\varphi)\setminus\D_{f-i-1}(\varphi)$, the map $p_{|Y1}:Y_1\to X$ is a locally trivial $\Gr(e-f,e-f+i)$-bundle. Therefore, the pair $(Y_1,p_{|Y_1})$ is a stratified Grassmann bundle with generic fiber a point on $X$.
        The dimension of the preimage of $X_i$ is
        \[
            \dim(X)-i(e-f+i)+i(e-f)=\dim(X)-i^2.
        \]
        Therefore the dimension of the induced stratification on $Y_1$ is less than $\dim(X)$, i.e.\ the dimension of the open stratum. Thus $\dim(Y_1)\leq\dim(X)$. Notice that $Y_1$ is the zero locus of a global section of a vector bundle of rank $f$ on a smooth variety $\bbGr_X(e-f,\mE)$ of dimension $\dim(X)+f$. It follows that $Y_1$ is Cohen--Macaulay and irreducible.
        In particular,
        \[
            p_{|Y_1}:Y_1\to X
        \]
        is a birational map, with exceptional locus $\D_{f-1}(\varphi)$.

        For $(ii)$, consider the stratification
        \[
            X=\D_{f}(\varphi)\supset\D_{f-1}(\varphi)\supset\D_{f-2}(\varphi)\supset\dots
        \]
        Then, over $X_i=\D_{f-i}(\varphi)\setminus\D_{f-i-1}(\varphi)$, the map $p_{|Y_2}:Y_2\to X$ is a locally trivial $\PP^{e-f+i-1}$-bundle. Therefore the pair $(Y_2,p_{|Y_2})$ is a stratified projective bundle with generic fiber $\PP^{e-f-1}$ on $X$. The dimension of the preimage of $X_i$ is 
        \[
            \dim(X)-i(e-f+i)+(e-f+i-1)=\dim(X)+(e-f-1)-i(e-f+i-1).
        \]
        For $i\geq 1$, one has $i(e-f+i-1)>1$, hence the dimension of the induced stratification on $Y_2$ is less than $\dim(X)+(e-f+1)$, i.e.\ the dimension of the open stratum. Thus, $\dim(Y_2)\leq\dim(X)+(e-f+1)$.
        On the other hand, $Y_2$ is the zero locus of a global section of a vector bundle of rank $f$ on the smooth variety $\PP_X(\mE)$ of dimension $\dim(X)+(e-1)$. It follows that $Y$ is Cohen--Macaulay and irreducible. 

        We now study $Z$. Using the same argument as above, we have a stratification of $X$ as $X=\bigsqcup X_i$, with $X_i=\D_{f-i}(\varphi)\setminus\D_{f-i-1}(\varphi)$. For $x\in X_i$, we have that $q_{|Z}^{-1}(x)$ is $\PP(\ker\varphi_x^T)$. In particular, if $x\in X_0$ then the fiber is empty, while if $x\in X_i$ for $i\geq1$, the fiber is $\PP^{i-1}$. Thus, the pair $(Z,q_{|Z})$ is a stratified projective bundle on $\D_{f-1}(\varphi)$, with general fiber a point. In particular, for $i\geq1$, the dimension of the preimage of $X_i$ is  
        \[
              \dim(X)-i(e-f+1)+(i-1)=\dim(\D_{f-1}(\varphi))+(1-i)(e-f+i).
        \]
        Hence, for $i\geq2$, the dimension of every stratum of the induced stratification of $Z$ is less than $\dim(\D_{f-1}(\varphi))$, therefore $Z\leq\D_{f-1}(\varphi)=\dim(X)-(e-f+1)$. On the other hand, $Z$ is the zero locus of a global section of a vector bundle of rank $e$ on a smooth variety $\PP_X(\mF^{\vee})$ of dimension $\dim(X)+f-1$, it follows that $Z$ is Cohen--Macaulay and irreducible of dimension $\dim(X)-(e-f+1)$. In particular, $Z$ is birational to $\D_{f-1}(\varphi)$.
       
        Finally, we prove $(iii)$. If $f=e-1$, we are in the hypothesis of \cite[Lemma 2.1]{kuznetsovKuchle}, so the first part follows. It remains to study $Z$. But this follows from point $(ii)$.
\end{proof}

As a consequence of the previous lemma, we can prove a result which solves the Fano-visitor problem for a significant class of examples.

\begin{proposition}\label[proposition]{prFanoHost}
  Let $X$ be a smooth variety. Let $\mE$, $\mF$ be vector bundles of ranks $\rk\ \mE=e>f=\rk\ \mF$ such that $\mE^{\vee}\otimes\mF$ is globally generated, and let $\varphi:\mE\to\mF$ be a general morphism of vector bundles. Let $p\colon\PP_X(\mE)\to X$ and $q\colon\PP_X(\mF^{\vee})\to X$  the natural projections, and consider
    \[
        \varphi_{\mE}\in H^0(\PP_X(\mE),p^*\mF\otimes\of_{\mE}(1)),\ \varphi_{\mF^{\vee}}\in H^0(\PP_X(\mF^{\vee}),q^*\mE^{\vee}\otimes\of_{\mF^{\vee}}(1))
    \] corresponding to $\varphi$.
    Define the varieties
    \[
        Y=V(\varphi_{\mE})\subset\PP_X(\mE), \ \  Z=V(\varphi_{\mF^{\vee}})\subset\PP_X(\mF^{\vee}).
    \]
 
    Suppose the following conditions hold:
    \begin{enumerate}
        \item[\textnormal{a.}] $\D_{f-2}(\varphi)=\emptyset$;
        \item[\textnormal{b.}] $p^*(K_X\otimes \det\mE\otimes \det\mF)\otimes\of_{\mE}(-e+f)$ is anti-ample;
    \end{enumerate}
    Then $Y$ is a Fano host for $Z$.
\end{proposition}
\begin{proof}
    By \cref{prop:degSwap}(ii), we have that $Y$ is a stratified projective bundle on $X$ with generic fiber $\PP^{e-f-1}$, which changes to a $\PP^{e-f}$ on $\D_{f-1}(\varphi)$. By \cref{prop:degSwap}(ii), $\D_{f-1}(\varphi)$ is birational to $Z$ and, since $\D_{f-2}(\varphi)$ is empty for (a), they are isomorphic.

    We want to prove that there exists an embedding $\DC(Z)\to\DC(Y)$. 
    Notice that on $\D_{f-1}(\varphi)$ the projection $p_{|Y}$ is a locally trivial $\PP^{e-f}$-bundle, hence \[\dim (Z\times_XY)=\dim(Z)+(e-f)=\dim(X)-1.\]  
    Thus we can apply \cite[Theorem 8.8]{Kuz07}, and obtain
    \[
        \DC(Y)=\langle\DC(Z),p^*\DC(X)\otimes\of_{\mE}(1),\ldots,p^*\DC(X)\otimes\of_{\mE}(e-f)\rangle.
    \]

    Note that the canonical bundle of $\PP_X(\mE)$ is
    \[
        K_{\PP_X(\mE)}=p^*(K_X\otimes\det\mE)\otimes\of_{\mE}(-e),
    \]
    with $p\colon\PP_X(\mE)\to X$ the natural projection.
    By the adjunction formula and condition (b), the anticanonical bundle of $Y$ is the restriction of an ample line bundle, hence is ample. Therefore, $Y$ is a Fano variety. 
\end{proof}

\begin{rmk}
    If $X$ is a smooth Fano variety with Picard rank one, for example a Grassmannian, then the condition (b) in \cref{prFanoHost} becomes a numerical condition. In fact, let $\iota_X$ be the index of $X$, and let $\det\mE\cong\of_X(\alpha)$, $\det\mF\cong\of_X(\beta)$. Then the condition (b) can be rewritten as: $\iota_X-\alpha-\beta>0$ and $f<e$.
\end{rmk}

We obtain in this way a new set of varieties for which the Fano-visitor problem has a positive answer. As an application, we show an improvement to the Fano dimension of a general Enriques surface.

\section{A general Enriques surface and its host} \label[section]{section3}

In this final section, we discuss two Fano-hosts for a general Enriques surface: the first one is six-dimensional and described by Kuznetsov in \cite{K18}; the second one is four-dimensional, and it is the main result of this paper.

We fix the following notation: let $V_3,\ V_3',\ W_3$ be 3-dimensional vector spaces, $X=\PP(V_3)\times\PP(V_3')$, $\mF=\of_{X}(2,0)\oplus\of_X(0,2)$, $\mE=\of_X\otimes W_3$, and $q\colon\PP_X(\mF^{\vee})\to X$ the natural projection to $X$. 
 
Recall from \cite[Lemma 2.1]{OS20} that a general Enriques surface can be described as the first degeneracy loci $D_1(\varphi)$  of a general morphism:
\[
    \varphi:\mE\to\mF.
\]

By isomorphism (\ref{eq:isovb}) 

\begin{equation*} 
    H^0(X,\mE^{\vee}\otimes\mF)\cong H^0(\PP_X(\mF^{\vee}),q^*\mE^{\vee}\otimes\of_{\mF^{\vee}}(1)),
\end{equation*} 

Denote with $\varphi_{\mF^{\vee}}$ the image of $\varphi$ via the isomorphism (\ref{eq:isovb}).

Hence, we can consider the zero locus $S=V(\varphi_{\mF^{\vee}}) \subset \PP_X(\mF^{\vee})$, which can be written as
\[
    S=\mZ(q^*\mE^{\vee}\otimes\of_{\mF^{\vee}}(1))\subset\PP_X(\mF^{\vee})
\] the Enriques surface considered in \cite[Lemma 3]{K18}. Notice that $q^*\mE^{\vee}\otimes\of_{\mF^{\vee}}(1)=\of_{\mF^{\vee}}(1)^{\oplus 3}$, and varying the three-dimensional subspace of global sections of $\of_{\mF^{\vee}}(1)$, one obtains the general Enriques surface. We will show in \cref{thm:main} that $S$ and $D_1(\varphi)$ are in fact isomorphic.

\subsection*{Six-dimensional Fano-host.} In \cite{K18}, in order to produce a six-dimensional Fano-host for $S$, Kuznetsov considered the product $\PP_{X}(\mF^{\vee})\times\PP(W_3)$. The key observation is the isomorphism given by the K\"unneth formula
\begin{equation} \label{eq:kun1}
 H^0(\PP_X(\mF^{\vee}),q^*\mE^{\vee}\otimes\of_{\mF^{\vee}}(1)) \cong H^0(\PP_X(\mF^{\vee})\times\PP(W_3),\of_{\mF^{\vee}}(1)\boxtimes\of_{\PP(W_3)}(1)).\end{equation}

Hence, the same section $\varphi$ from the previous paragraph is associated to a unique general global section $\widetilde\varphi$ of $ H^0(\PP_X(\mF^{\vee})\times\PP(W_3),\of_{\mF^{\vee}}(1)\boxtimes\of_{\PP(W_3)}(1))$. We can then consider its zero locus $V(\widetilde\varphi)=T \subset \PP_X(\mF^{\vee})\times\PP(W_3)$, i.e.
 
\[
    T=\mZ(\of_{\mF^{\vee}}(1)\boxtimes\of_{\PP(W_3)}(1))\subset \PP_X(\mF^{\vee})\times\PP(W_3).
\] 

In \cite[Theorem 4]{K18} is proved that $T$ is a Fano sixfold, described as a stratified projective bundle bundle on $\PP_X(\mF^{\vee})$, with general fiber a $\PP^1$ which jumps to a $\PP^2$-bundle over $S$. From this, it follows that
\[      \DC(T)=\langle\DC(S),\DC(\PP_X(\mF^{\vee})),\DC(\PP_X(\mF^{\vee}))\rangle.
\]

\subsection*{Four-dimensional Fano-host.}
The key idea behind this paper is to consider another variety $Y$ which is defined by the very same $\varphi$, and to relate it with $S$. In fact, as in (\ref{eq:kun1}), one gets another canonical isomorphism

\begin{equation}\label{eq2}
H^0(\PP_X(\mF^{\vee})\times\PP(W_3),\of_{\mF^{\vee}}(1)\boxtimes\of_{\PP(W_3)}(1))\cong H^0(\PP(W_3)\times X,\of_{\PP(W_3)}(1)\boxtimes\mF).
\end{equation} 

Denote with $\widehat\varphi$ the unique global section of $\of_{\PP(W_3)}(1)\boxtimes\mF$ associated to $\widetilde{\varphi}$ via (\ref{eq2}). Therefore, the zero locus $V(\widehat\varphi)=Y \subset \PP(W_3)\times X $ defines a variety 
\[
    Y=\mZ(\of_{\PP(W_3)}(1)\boxtimes\mF)\subset\PP(W_3)\times X,
\]
 which can be rewritten as 
\[
    Y=\mZ(\of(1,0,2)\oplus\of(1,2,0))\subset\PP(W_3)\times\PP(V_3)\times\PP(V_3').
\]

In the following, we prove that $X$ and $Y$ are birational, and how $Y$ can be seen as a Fano-host for $S$.

\begin{thm} \label[thm]{thm:main}
The variety
\[
    Y=\mZ(\of(1,0,2)\oplus\of(1,2,0))\subset\PP(W_3)\times\PP(V_3)\times\PP(V_3')
\] is a Fano fourfold. Moreover, $Y\cong\Bl_{S}(\PP(V_3)\times\PP(V_3'))$, with $S$ a general Enriques surface.
\end{thm}  
\begin{proof}

In this proof, in agreement with the above notation, we set $X= \PP(V_3)\times\PP(V_3')$, $\mE= \of_X \otimes W_3$, $\mF=\of_X(2,0) \oplus \of_X(0,2)$.

 As in the previous discussion, we start from a general morphism of vector bundles on $X$,
        $\varphi:\mE\to\mF.$
     From the isomorphism in (\ref{eq:isovb}), denote $\varphi_{\mF^{\vee}}$ the global section of $q^*\mE \otimes \of_{\mF^{\vee}}(1)$ asocciated to $\varphi$.
    \[
        \varphi_{\mF^{\vee}}\in H^0(\PP_{X}(\mF^{\vee}),q^*\mE^{\vee} \otimes \of_{\mF^{\vee}}(1)).
    \]
  
    By \cref{thm:ddl}, $\D_{1}(\varphi)\subset X$ is a smooth surface, because $\D_{0}(\varphi)$ is empty. By \cref{prop:degSwap}(iii) $\D_{1}(\varphi)$ is birational to
    \[
        S=V(\varphi_{\mF^{\vee}})\subset\PP_{X}(\mF^{\vee}).
    \]
    
    $S$ is the Enriques surface $S$ described in \cite[Lemma 3]{K18}. Notice that $\D_{0}(\varphi)$ is empty, thus we conclude that $\D_1(\varphi) \cong S$.

    We now turn our attention to the variety $Y$. The latter is defined as the zero locus of a section of a vector bundle of rank 2 in $\PP^2\times\PP^2\times\PP^2$, thus $Y$ is four-dimensional. By the adjunction formula, 
    \[
        K_Y=K_{\PP^2\times\PP^2\times\PP^2}\otimes\of(1,2,0)\otimes\of(1,0,2)_{|Y}=\of_{\PP^2\times\PP^2\times\PP^2}(-1,-1,-1)_{|Y}.
    \]
    Hence, $Y$ is Fano.

    Notice that since $\mE=W_3\otimes\of_{X}$, then $\PP(W_3)\times X=\PP_X(\mE)$. Thus, as before, if $p\colon\PP_{X}(\mE)\to X$ is the natural projection, then $Y$ can be rewritten as 
    \[
        Y=V(\varphi_{\mE})\subset\PP_{X}(\mE),
    \]
    where $\varphi_{\mE}$ is the global section of $p^*\mF \otimes \of_{\mE}(1)$ associated to $\varphi$ via the isomorphism (\ref{congPD1}).
Therefore, we are in the setting of \cref{prop:degSwap}. By \cref{prop:degSwap}(i), $Y$ is birational to $X$.
In particular, since $f=e-1=2$, and since we have already shown that $D_{1}(\varphi) \cong S$, by \cref{prop:degSwap}(iii) we have the isomorphism
    \[
    Y\cong\Bl_{S}(\PP^2\times\PP^2).\]

\end{proof}

We are now in the position to describe the numerical invariants and the categorical properties of $Y$.

Since $Y$ is a complete intersection, its invariants can be quickly computed using standard exact sequences and the Riemann-Roch theorem.

\begin{corollary}
    The variety $Y$ has the following invariants:
    \[
        e(Y)=21,\ h^0(-K_Y)=27,\ (-K_Y)^4=102.
    \]
    The half-lower Hodge diamond of $Y$ is 
    \[
\begin{array}{cccccccccc}
	
	0 & & 0 & & 13 & & 0 & & 0 \\
	& 0 & & 0 & & 0 & & 0 & \\
	& & 0 & & 3 & & 0 & & \\
	& & & 0 & & 0 & & & \\
	& & & & 1 & & & & \\
\end{array}
\]
Moreover, if $\Lambda\subset H^2(S,\Z)$ is the  rank ten lattice associated to the Enriques surface $S$, the integral cohomology of $Y$ is:
\[
    H^k(Y,\Z)=\begin{cases}
        \Z & \text{ if ~$k=0,\ 8$}\\
        \Z^{\oplus 3} & \text{ if ~$k=2,\ 6$}\\
        \Z^{\oplus 3}\oplus\Lambda\oplus\Z/2\Z & \text{ if ~$k=4$}\\
        0 & \text{otherwise}
    \end{cases}
\]
\end{corollary}

\begin{corollary} \label[corollary]{cor:main}
    The variety $Y$ is a fourfold with a semiorthogonal decomposition
    \[
        \DC(Y)=\langle\DC(S),E_1,\ldots,E_9\rangle
    \]
    where $E_1,\ldots,E_9$ are exceptional bundles. The Hodge diamond of $Y$ is diagonal, but $K^0(Y)$ contains a 2-torsion class; in particular, $\DC(Y)$ does not have a full exceptional collection.
\end{corollary}
\begin{proof}
    The semiorthogonal decomposition is given by Orlov’s blow-up formula and the fact that $\DC(\PP^2\times\PP^2)$ is generated by an exceptional collection of length nine. 

    The Grothendieck group is additive with respect to semiorthogonal decompositions, hence 
    \[
        K_0(Y)=K_0(S)\oplus\mathbb{Z}^{\oplus 9}.
    \]
    Notice that the 2-torsion in $S$ induces a 2-torsion in $Y$, hence, by \cite[Lemma 1]{K18}, $\DC(Y)$ does not have a full exceptional collection.
\end{proof}
This yields a four-dimensional variety as a solution to the Fano-visitor problem for Enriques surfaces general in moduli.

With \cref{cor:main} we proved that the Fano dimension of $S$ is at most $4$. Observe that neither $\PP^1$ nor a del Pezzo surface can be a Fano-host for $S$. Therefore, if one shows that no Fano threefolds can be a Fano host for $S$, it follows that the Fano dimension of a general Enriques surface is 4. 

Assume by contradiction that $Y$ is a three-dimensional Fano host for $S$. Then we obtain the following inequality on the dimensions of Hochschild homology:
\[
    12=\dim (\text{HH}_0(S))\leq (\dim\text{HH}_0(Y)).
\]
If $\rho(Y)$ denotes the Picard rank of $Y$, we have \[\dim\text{HH}_0(Y)=2+\rho(Y),\] and hence $\rho(Y)\geq 5$. By classification of Fano threefolds, there are only eight families of Fano threefolds with Picard rank $\geq5$, and all of them admit a full exceptional collection. This implies that $K_0(Y)$ is torsion-free, which contradicts the assumption that $K_0(S)\subset K_0(Y)$.

\frenchspacing

\newcommand{\etalchar}[1]{$^{#1}$}

\end{document}